\documentclass[10pt]{amsproc}
\usepackage{amssymb}
\usepackage{graphicx,color}
\usepackage{amsmath}
\usepackage{amscd} 

\theoremstyle{plain}
 \newtheorem{thm}{Theorem}[section]
 \newtheorem{prop}{Proposition}[section]
 \newtheorem{lem}{Lemma}[section]
 
 \newtheorem{claim}{Claim}[section]
\theoremstyle{definition}
 
 \newtheorem{dfn}{Definition}[section]
\theoremstyle{remark}
 \newtheorem{rem}{Remark}[section] 
 \numberwithin{equation}{section}

\renewcommand{\setminus}{\smallsetminus}
\newcommand{\pos}{\mathrm{pos}}

\newcommand{\supp}{\mathrm{supp} \, }
\newcommand{\conv}{\mathrm{conv}}
\title[Sums of two self-similar Cantor sets]
{Sums of two self-similar Cantor sets}

\author[Y.\ Takahashi]{YUKI TAKAHASHI}

\address{Department of Mathematics, Bar-Ilan University, Ramat-Gan, 5290002, Israel}

\email{takahashi@math.biu.ac.il}

\thanks{Y.\ T. \ was supported by the Israel Science Foundation grant 396/15 (PI: B.\ Solomyak).}

\date{today}

\begin{document}

\vspace{18mm}
\setcounter{page}{1}
\thispagestyle{empty}

\begin{abstract}
We show that for any pair of self-similar Cantor sets with sum of Hausdorff dimensions greater than $1$, 
one can create an interval in the sumset by applying arbitrary small perturbations 
(without leaving the class of self-similar Cantor sets). 
In our setting the perturbations have more freedom than in the setting of the Palis' conjecture, so 
our result can be viewed as an affirmative answer to a weaker form of the Palis' conjecture. 
\end{abstract}

\maketitle

\section{Introduction and main results}
\subsection{Sums of two Cantor sets}
Sums of two Cantor sets arise naturally in dynamical systems (e.g., \cite{Newhouse}, \cite{PalisTakens}), 
in number theory (e.g., \cite{Hall}, \cite{Moreira00}) and also in spectral theory (e.g., \cite{DG1}, \cite{DG}). 
In 1970's, 
Palis conjectured that for generic pairs of dynamically defined 
Cantor sets their sumset contains an interval 
if the sum of their Hausdorff dimensions is greater than $1$ (see, e.g., \cite{PalisTakens}).  
For nonlinear Cantor sets this question was proven in \cite{Moreira3}. 
The problem is still open for affine Cantor sets.

In \cite{Takahashi}, by relying on the techniques invented by Moreira and Yoccoz in \cite{Moreira3},  
the author showed that for any pair of homogeneous Cantor sets  
one can create an interval in the sumset by applying arbitrary small perturbations, if 
the sum of their Hausdorff dimensions is greater than $1$. 
In this paper we extend the result of \cite{Takahashi} 
to any pair of self-similar Cantor sets. 
The idea of the proof is borrowed from \cite{Moreira3}, \cite{Takahashi2} and \cite{Takahashi}.

\subsection{Self-similar Cantor sets and main results}\label{section2}

Write $I = [0, 1]$.

\begin{dfn}\label{selfsimilar}
We call $K \subset \mathbb{R}$ a \emph{self-similar Cantor set} if the following holds: 
there exists a finite alphabet $\mathcal{A}$ and a set of linear contractions $\mathcal{F} = \{ f_a \}_{a \in \mathcal{A}}$ 
on $\mathbb{R}$ 
such that 
\begin{itemize}
\item[(i)] $K = \displaystyle{ \bigcup_{a \in \mathcal{A}} f_a(K) }$;
\item[(ii)] $f_a( \conv(K) ) \ (a \in \mathcal{A})$ are pairwise disjoint, 
\end{itemize}
where $\conv(K)$ is the convex hull of $K$. 
Without loss of generality, we can further assume that $K \subset I$.  
For $a \in \mathcal{A}$, we denote $f_a(I)$ by $I(a)$. 
\end{dfn}

\begin{rem}
Definition \ref{selfsimilar} is not the most standard. 
Self-similar Cantor set is normally defined as a set $K$ together with a set of contracting maps that generates $K$. 
\end{rem}

Let $I_1, I_2 \subset \mathbb{R}$ be closed intervals.  
We say that $I_1$ and $I_2$ are $\epsilon$-close if 
\begin{itemize}
\item[(i)] 
$\displaystyle{ 1 - \epsilon <  \frac{ | I_1 | }{ | I_2 | } < 1 + \epsilon }$;  
\item[(ii)] the distance between the center of $I_1$ 
and $I_2$ is less than $\displaystyle{ \frac{\epsilon}{\min \{ | I_1 |, | I_2 | \} }  }$. 
\end{itemize}

\begin{dfn}
Let $K, \widetilde{K}$ be self-similar Cantor sets. We say that $K$ and $\widetilde{K}$ are \emph{$\epsilon$-close} 
if the following holds: 
there exist 
sets of contracting similarities $\mathcal{F} = \{f_a\}_{a \in \mathcal{A}}$
 (resp. $\widetilde{\mathcal{F}} = \{\tilde{f}_{\tilde{a}}\}_{\tilde{a} \in \widetilde{ \mathcal{A} } }$) 
that generate $K$ (resp. $\widetilde{K}$) such that 
\begin{itemize}
\item[(i)] $\mathcal{A} = \widetilde{\mathcal{A}}$; 
\item[(ii)] $I(a)$ and $\tilde{I}(a)$ are $\epsilon$-close for all $a \in \mathcal{A}$. 
\end{itemize}
\end{dfn}

Our main results are the following: 

\begin{thm}\label{thm1}
Let $K$, $K'$ be self-similar Cantor sets such that the sum of their Hausdorff dimensions is greater than $1$.  
Then, for every $\epsilon > 0$ and $M > 0$, 
there exists a self-similar Cantor set $\widetilde{K}$ and a set $E \subset (M^{-1}, M)$ 
such that 
\begin{itemize}
\item[(i)] $\left| (M^{-1}, M) \setminus E \right| < \epsilon$; 
\item[(ii)] $\widetilde{K}$ is $\epsilon$-close to $K$; 
\item[(iii)] $\widetilde{K} + r K'$ contains an interval for all $r \in E$. 
\end{itemize}
\end{thm}

Analogous result holds for sums of Cantor sets with itself: 

\begin{thm}\label{thm2}
Let $K$ be a self-similar Cantor set with Hausdorff dimension greater than $1/2$. 
Then, for every $\epsilon > 0$ and $M > 0$, 
there exists a self-similar Cantor set $\widetilde{K}$ and a set $E \subset (M^{-1}, M)$ such that 
\begin{itemize} 
\item[(i)] $\left| (M^{-1}, M) \setminus E \right| < \epsilon$; 
\item[(ii)] $\widetilde{K}$ is $\epsilon$-close to $K$;
\item[(iii)] $\widetilde{K} + r \widetilde{K}$ contains an interval for all $r \in E$. 
\end{itemize}
\end{thm}

\begin{dfn}
We call $K \subset \mathbb{R}$ a \emph{homogeneous Cantor set} if the following holds: 
there exists a finite alphabet $\mathcal{A}$ and a set of linear contractions $\mathcal{F} = \{ f_a \}_{a \in \mathcal{A}}$ 
on $\mathbb{R}$ 
such that 
\begin{itemize}
\item[(i)] $K = \displaystyle{ \bigcup_{a \in \mathcal{A}} f_a(K) }$;
\item[(ii)] all $f_a \ (a \in \mathcal{A})$ have the same contracting ratio; 
\item[(ii)] $f_a( \conv(K) ) \ (a \in \mathcal{A})$ are pairwise disjoint.  
\end{itemize}
\end{dfn}

In \cite{Takahashi}, the author proved the following: 
\begin{thm}\label{mukashi}
Let $K$, $K'$ be homogeneous Cantor sets 
such that the sum of their Hausdorff dimensions is greater than $1$. 
Assume that there exist 
sets of contracting similarities $\mathcal{F} = \{f_a\}_{a \in \mathcal{A}}$
 (resp. $\mathcal{F}' = \{ f'_{a'} \}_{ a' \in \mathcal{A}' }$) 
that generate $K$ (resp. $K'$) 
and the following holds: 
\begin{itemize}
\item[(a)] the contracting ratio of $f_a \ (a \in \mathcal{A})$ 
is equal to the contracting ratio of $f'_{a'} \ (a' \in \mathcal{A}' )$;  
\item[(b)] $\mu \ast \mu'$ has $L^2$-density, where $\mu$ (resp. $\mu'$) is the 
uniform probability self-similar measure of $K$ (resp. $K'$). 
\end{itemize}
Then, for every $\epsilon > 0$ there exists a homogeneous Cantor set 
$\widetilde{K}$ such that 
\begin{itemize}
\item[(i)] $\widetilde{K}$ is $\epsilon$-close to $K$; 
\item[(ii)] $\widetilde{K} + K'$ contains an interval.  
\end{itemize}
\end{thm}

Theorem \ref{thm1} is an extension of Theorem \ref{mukashi}. 
In Theorem \ref{thm1}, $K$ and $K'$ are general self-similar Cantor sets (not necessarily homogeneous) and 
the assumptions (a) and (b) are dropped. 
Furthermore, $\widetilde{K} + r K'$ contains an interval for ``many r".

\subsection{Structure of the paper} 
In section \ref{recurrence} we define recurrent sets and renormalization operators, 
and describe the basic idea of the proof.   
The outline of the proof of Theorem \ref{thm1} is given in section \ref{outline}.  
In section \ref{construct} we will construct the set $\mathcal{L}$ 
which is the candidate of a recurrent set.  
In section \ref{key_prop} we will prove the key proposition, 
which roughly claims that with ``very high probability" any point in the set $\mathcal{L}$  
can return to itself by an action of a renormalization operator.

\section{Renormalizations and recurrent sets}\label{recurrence}

\subsection{Projections of $K \times r K'$}
Let $K$, $K'$ be self-similar Cantor sets. Since $K + K' = K - (-K')$ and $-K'$ is again a self-similar Cantor set, 
from below we consider only differences of self-similar Cantor sets, instead of sums.  
Let $\Pi$ be the projection of $\mathbb{R}^2$ onto the 
$y$-axis along the lines that make the angle $\pi/4$ with the $x$-axis. 
Then, it is easy to see that 
\begin{equation*}
K - r K' \text{ contains an interval } \iff \Pi ( K \times r K') \text{ contains an interval}. 
\end{equation*}

\subsection{Renormalizations} 

Throughout this section, we fix self-similar Cantor sets $K, K' \subset I$   
and sets of contracting similarities 
$\mathcal{F} = \{ f_a \}_{a \in \mathcal{A}}$ (resp. $\mathcal{F}' = \{ f'_{a'} \}_{a' \in \mathcal{A}'}$) 
that generate $K$ (resp. $K'$). 
Denote $\mathcal{A}^* = \cup_{n \geq 1} \mathcal{A}^n$ and 
$\mathcal{A}'^{*} = \cup_{n \geq 1} \mathcal{A}'^n$. 
Let $\mathcal{P}$ be the set of all linear transformations from $I$ to $\mathbb{R}$ that have positive linear coefficient.  
Call a pair $(h \times h', \ell)$ a \emph{configuration}, where 
$h \times h' \in \mathcal{P}^2$ and $\ell$ is a line in $\mathbb{R}^2$ that has slope $1$. 
We define an equivalence relation on the set of configurations in the following way: 
\begin{equation*}
\begin{aligned}
&( h_1 \times h'_1, \ell_1 ) \sim  ( h_2 \times h'_2, \ell_2 ) \\
\iff &\text{$\exists$ homothety $g: \mathbb{R}^2 \to \mathbb{R}^2$ s.t. } 
g \circ h_1 = h_2, \, g \circ h'_1 = h'_2 \text{ and } g( \ell_1 ) = \ell_2, 
\end{aligned}
\end{equation*}
where $g: \mathbb{R}^2 \to \mathbb{R}^2$ is a homothety if 
$g(x) = r x + t$ for some $r > 0$ and $t \in \mathbb{R}^2$.   
Let $Q$ be the quotient of configurations by the above equivalence relation. 
Let $u \in Q$, and let 
$( h \times h', \ell )$ be the configuration that satisfies $u = [ ( h \times h', \ell ) ]$, 
$(h \times h')(O) = O$ and $h \equiv \mathrm{id}$. 
Consider the map  
\begin{equation}\label{identification0}
\begin{aligned}
Q &\to \mathbb{R}^2 \\
u &\mapsto ( r, t ), 
\end{aligned}
\end{equation}
where $r = | h'(I) | $ and $t$ is the y-coordinate of the y-intercept of the line $\ell$. 
It is easy to see that this map is a bijection. From below we use this identification freely. 
For $\underline{a} = a_1 \cdots a_n \in \mathcal{A}^*$ and 
$\underline{a}' = a'_1 \cdots a'_{n'} \in \mathcal{A}'^{*}$, we define 
$T_{ \underline{a} } T'_{ \underline{a}' } (\cdot): Q \to Q$ by 
\begin{equation}\label{renormalization}
T_{ \underline{a} } T'_{ \underline{a}' } ( [ ( h \times h', \ell ) ] ) 
= \big[ \big( ( h \times h' ) \circ ( f_{a_1} \circ \cdots \circ f_{a_n} \times f'_{a'_1} \circ \cdots \circ f'_{a'_n} ), \ell \big) \big], 
\end{equation}
and call this map a \emph{renormalization operator}.

\subsection{Recurrent sets}
Let $u \in Q$, and let $(h \times h', \ell)$ be a configuration such that $u = [(h \times h', \ell)]$. 
We say that $u$ is \emph{intersecting} if $(h \times h')(K \times K') \cap \ell \neq \emptyset$.

\begin{lem}\label{crucial}
Let $u \in Q$. Then $u$ is intersecting if and only if 
there exists $M > 0$, $\underline{a}_i \in \mathcal{A}^*$ and  
$\underline{a}'_{i} \in \mathcal{A}'^* \ (i = 1, 2, \cdots)$ 
and the following holds: 
let  
$\{ u_i \} \ (i = 0, 1, \cdots)$ be the sequence defined by 
\begin{equation}\label{u}
u_0 = u, \ u_i = T_{ \underline{a}_i } T'_{ \underline{a}'_i } u_{i-1}. 
\end{equation} 
Then, writing $u_i = (r_i, t_i)$, we have 
$M^{-1} < r_i < M$ and $| t_i | < M \ (i = 0, 1, 2, \cdots)$. 
\end{lem}

\begin{proof}
Assume first that $u$ is intersecting. 
Let $(h \times h', \ell)$ be a configuration such that $u = [(h \times h', \ell)]$. 
Let $x \in ( h \times h' )( K \times K' ) \cap \ell$, and let $M > 0$ be a sufficiently large constant. 
Then, it is easy to see that there exit 
$\underline{a}_i \in \mathcal{A}^*, \, \underline{a}'_i \in \mathcal{A}'^* \ (i = 1, 2, \cdots)$ such that 
\begin{equation*}
x = \bigcap_{i=1}^{\infty} ( h \times h' ) \circ 
( f_{ \underline{a}_1 } \times f'_{ \underline{a}'_1 } ) \circ \cdots \circ 
( f_{ \underline{a}_i } \times f'_{ \underline{a}'_i } ) (I^2)
\end{equation*} 
and 
\begin{equation}\label{ratio1}
M^{-1} < \frac{ \big| h' \circ f'_{ \underline{a}'_1 } \circ \cdots \circ f'_{ \underline{a}'_i } (I) \big| }
{ \big| h \circ f_{ \underline{a}_1 } \circ \cdots \circ f_{ \underline{a}_i } (I) \big| } < M. 
\end{equation}
Define $\{ u_i \}$ by (\ref{u}).  
Note that 
\begin{equation*}
u_i = \big[ \big(  (h \times h') \circ 
( f_{ \underline{a}_1 } \times f'_{ \underline{a}'_{1} } ) \circ \cdots \circ 
( f_{ \underline{a}_i } \times f'_{ \underline{a}'_i } ) , \ell  \big) \big]. 
\end{equation*} 
By (\ref{ratio1}), we have $M^{-1} < r_i < M$. 
Since $x \in \ell$, we have 
\begin{equation*}
( h \times h' ) \circ ( f_{ \underline{a}_1 } \times f'_{ \underline{a}'_1 } ) 
\circ \cdots \circ ( f_{ \underline{a}_i } \times f'_{ \underline{a}'_i } ) (I^2) \cap \ell \neq \emptyset. 
\end{equation*}
This implies that 
$-1 < t_i < M$. 

Assume next that $u$ is not intersecting. 
Let us take $\underline{a}_i \in \mathcal{A}^*, \, 
\underline{a}'_i \in \mathcal{A}'^* \ (i = 1, 2, \cdots)$ and $M > 0$.
Let $\{ u_i \}$ be the sequence 
defined by (\ref{u}). Assume that we have $M^{-1} < r_i < M$. 
Write 
\begin{equation*}
x = \bigcap_{i=1}^{\infty} ( h \times h' ) \circ 
( f_{ \underline{a}_1 } \times f'_{ \underline{a}'_1 } ) \circ \cdots \circ 
( f_{ \underline{a}_i } \times f'_{ \underline{a}'_i } ) (I^2). 
\end{equation*}
Since $x \notin \ell$, 
we have 
\begin{equation*}
( h \times h' ) \circ ( f_{ \underline{a}_1 } \times f'_{ \underline{a}'_1 } ) 
\circ \cdots \circ ( f_{ \underline{a}_i } \times f'_{ \underline{a}'_i } ) (I^2) \cap \ell = \emptyset
\end{equation*}  
for sufficiently large $i$. Since $| r_i |$ is bounded, 
this implies that $\lim_{i \to \infty} | t_i | = \infty$. 
\end{proof}

The above lemma leads to the following definition: 

\begin{dfn}
We call a nonempty set $\mathcal{L} \subset Q$ a \emph{recurrent set} 
if the following holds: 
there exists $M > 0$ and for every $u = (r, t) \in \mathcal{L}$ we have 
\begin{itemize}
\item[(i)] $M^{-1} < r < M$ and $| t | < M$; 
\item[(ii)] there exist $\underline{a} \in \mathcal{A}^*$ and $\underline{a}' \in \mathcal{A}'^*$ 
such that $T_{ \underline{a} } T'_{ \underline{a}' } u \in \mathcal{L}$.  
\end{itemize}
\end{dfn}

Lemma \ref{crucial} implies the following: 

\begin{prop}
Let $\mathcal{L}$ be a recurrent set and let $r > 0$. 
If the set $\{ t : (r, t) \in \mathcal{L}  \}$ 
contains an interval, then $K - r K'$ contains an interval. 
\end{prop}

\section{Outline of the proof of the main theorem}\label{outline}

\subsection{Perturbation}
In this section, we discuss the outline of the proof of Theorem \ref{thm1}. 
Let $\epsilon > 0$. 
Let $K$, $K'$ be self-similar Cantor sets, and let 
$\mathcal{F} = \{ f_a \}_{a \in \mathcal{A}}$ (resp. $\mathcal{F}' = \{ f'_{a'} \}_{a' \in \mathcal{A}'}$) be sets of 
contracting similarities that generate $K$ (resp. $K'$). 
Denote the Hausdorff dimension of $K$ (resp. $K'$) by $d$ (resp. $d'$). 
Let $\rho > 0$ be a sufficiently small number. 
\begin{rem}
In the proof we use constants $c_{k} \ (k = 0, 1, \cdots, 10)$. 
They may depend on each other but can be taken independently of $\rho > 0$. 
\end{rem}
By retaking $\mathcal{F}, \mathcal{F}'$ if necessary, we can further assume that 
\begin{itemize}
\item[(i)] $c_0^{-1} \rho^{1/2} < | I(a) |, | I'(a') | < c_0 \rho^{1/2}$ for all $a \in \mathcal{A}, a' \in \mathcal{A}'$; 
\item[(ii)] there exist disjoint sets $\mathcal{A}^{l}$, $\mathcal{A}^{s} \subset \mathcal{A}$ that satisfy 
$| \mathcal{A}^l |, | \mathcal{A}^s | > | \mathcal{A} | / 3$ and 
$\min_{a \in \mathcal{A}^l } | I(a) | > \max_{a' \in \mathcal{A}'} | I'( a' ) |$, 
$\max_{ a \in \mathcal{A}^s } | I(a) | < \min_{a' \in \mathcal{A}'} | I'(a') |$. 
\end{itemize} 
Take $\mathcal{A}^l_1$, $\mathcal{A}^l_2 \subset \mathcal{A}^l$ in such a way that 
$\mathcal{A}^l_1 \sqcup \mathcal{A}^l_2 = \mathcal{A}^l$ and $| \mathcal{A}^l_1 | = | \mathcal{A}^l_2 | = | \mathcal{A}^l | / 2$. 
Choose $\mathcal{A}^s_1$, $\mathcal{A}^s_2 \subset \mathcal{A}^s$ analogously. 
Denote $\mathcal{A}_1 = \mathcal{A}_1^l \cup \mathcal{A}_1^s$ and 
$\mathcal{A}_2 = \mathcal{A}_2^l \cup \mathcal{A}_2^s$. 

Let $c_1 > 0$ be a sufficiently large constant, to be chosen later. 
Let $a \in \mathcal{A}_1$ and $\omega \in (-\epsilon, \epsilon) \times (-1, 1)$. Write 
$\omega = ( \gamma, \delta )$.  
Let $f_a^{\omega}$ be the contracting map that satisfies the following: 
\begin{itemize}
\item[(i)] $\displaystyle{ \frac{ | f^{ \omega }_a(I) | }{ | f_a(I) | } } = 1 + \gamma$; 
\item[(ii)] the center of $f^{\omega}_a(I)$ corresponds with 
the center of $f_a(I)$ shifted by $\delta c_1 \rho \in 
(-c_1 \rho, c_1 \rho)$.  
\end{itemize} 
Define 
\begin{equation*}
\Omega^{l} = \left(  (-\epsilon, \epsilon) \times (-1, 1)  \right)^{\mathcal{A}^{l}_1} \text{ and \ }
\Omega^{s} = \left(  (-\epsilon, \epsilon) \times (-1, 1)  \right)^{\mathcal{A}^{s}_1}. 
\end{equation*}
Write $\Omega = \Omega^l \times \Omega^s$. 
Let $\underline{\omega} = 
( \omega_a )_{a \in \mathcal{A}_1 } 
\in \Omega$, and 
denote $\omega_a = ( \gamma_a, \delta_a )$.  
We define 
\begin{equation*}
\mathcal{F}^{\underline{\omega}} = \{ f^{ \underline{\omega} }_a \}_{a \in \mathcal{A}}, 
\end{equation*}
a set of contracting maps, in the following way: 
\begin{equation*}
f^{\underline{\omega}}_a = 
\begin{cases}
f^{\omega_a}_a & \text{ if \ } a \in \mathcal{A}_1 \\ 
f_a & \text{ if \ } a \in  \mathcal{A}_2. 
\end{cases}
\end{equation*}
Let $K^{\underline{\omega}}$ 
be the self-similar Cantor set generated by $\mathcal{F}^{\underline{\omega}}$. 
Note that if $\rho > 0$ is sufficiently small, 
then $K^{\underline{\omega}}$ is $\epsilon$-close to $K$. 

Recall that we defined the renormalization operator in (\ref{renormalization}). 
We define the 
renormalization operator $T^{\omega}_a T'_{a'}$ in analogous way. 
For $\underline{\omega} = ( \omega_a )_{a \in \mathcal{A}_1 } \in \Omega$ 
and $a \in \mathcal{A}$, $a' \in \mathcal{A}'$,  
we define 
\begin{equation*}
T^{\underline{\omega}}_a T'_{a'} = 
\begin{cases}
T^{\omega_a}_a T'_{a'} & \text{ if \ } a \in \mathcal{A}_1 \\
T_a T'_{a'} & \text{ if \ } a \in \mathcal{A}_2. 
\end{cases}
\end{equation*}

\subsection{Outline of the proof}
In section \ref{construct}, we will construct the set 
$E \subset (M^{-1}, M)$, and the set $L(r ) \subset (-1, M)$ for all $r \in E$. 
Define 
\begin{equation*}
\mathcal{L}^0 = \left\{ (r, t) : r \in E, \, t \in L( r ) \right\}. 
\end{equation*}  
Let 
\begin{equation*}
\mathcal{L}^{1} = 
\left\{  (r, t) : \exists (r_0, t_0) \in \mathcal{L}^0 \text{ with } | r - r_0 | < \rho, | t - t_0 | < \rho \right\}
\end{equation*}
and 
\begin{equation*}
\mathcal{L} = 
\left\{  (r, t) : \exists (r_0, t_0) \in \mathcal{L}^0 \text{ with } | r - r_0 | < \rho / 2, | t - t_0 | < \rho / 2 \right\}. 
\end{equation*}
We show that $\mathcal{L}$ is a recurrent set for some $\underline{\omega} \in \Omega$.  
For $u \in \mathcal{L}^1$, we define $\Omega^0(u) \subset \Omega$ to be the set of all 
$\underline{\omega} \in \Omega$ such that the following holds: 
there exist $\underline{b} \in \mathcal{A}^2$, $\underline{b}' \in \mathcal{A}'^2$ and the image 
\begin{equation*}
T^{\underline{\omega}}_{ \underline{b}} T'_{ \underline{b}' } (u) = \hat{u}
\end{equation*}
satisfies $\hat{u} \in \mathcal{L}^0$. 
The following crucial estimate will be proven in section \ref{key_prop}. 

\begin{prop}\label{key_lem}
There exists $c_2 > 0$ such that for any $u \in \mathcal{L}^1$, 
\begin{equation*}
\mathbb{P} \left( \Omega \setminus \Omega^0(u) \right) \leq 
\exp \left( -c_2 \rho^{ -\frac{1}{2} (d + d' - 1) } \right). 
\end{equation*}
\end{prop}

The sets $E$ and $L(r)$ are 
constructed in such a way that Proposition \ref{key_lem} holds. 
Below we prove Theorem \ref{thm1} assuming 
Proposition \ref{key_lem}. 
In section \ref{construct} we construct $E$ and $L(r)$, 
and show that the measure of the set $L(r)$ is bounded away from 
zero uniformly. 
Combining all these properties we prove Proposition \ref{key_lem} in section \ref{key_prop}. \vspace{2mm}


We choose a finite $\rho^{ 5/2 }$-dense subset $\Delta$ of $\mathcal{L}^1$. 
Note that 
\begin{equation*}
| \Delta | \leq c_3 \rho^{ -5/2 } \cdot \rho^{ -5/2 } = c_3 \rho^{-5}. 
\end{equation*}
Now, if $\rho > 0$ is small enough, 
\begin{equation*}
c_{3} \, \rho^{-5} \exp \left( -c_2 \rho^{ -\frac{1}{2} ( d + d' - 1 ) } \right) < 1, 
\end{equation*}
and therefore we can find $\underline{\omega}_0 \in \Omega$ such that 
$\underline{\omega}_0 \in \Omega^0(u)$ for all $u \in \Delta$.

\begin{rem}
The above is saying that any $u \in \Delta$ can return to $\mathcal{L}^0$ by an 
action of the renormalization operator of the form 
$T^{\underline{\omega}_0}_{ \underline{b}} T'_{ \underline{b}' }$. 
\end{rem}

Theorem \ref{thm1} follows from the following claim: 

\begin{claim}
For $\underline{\omega}_0 \in \Omega$, the set $\mathcal{L}$ is a recurrent set.  
\end{claim}

\begin{proof}[proof of the claim]
Let $u \in \mathcal{L}$. Write $u = (r, t)$. 
Let $u_0 = (r_0, t_0) \in \Delta$ be such that 
$| r - r_0 | < \rho^{5/2}$ and $| t - t_0 | < \rho^{5/2}$. 
By the choice of $\underline{\omega}_0$, we have $\underline{\omega}_0 \in \Omega^{0}( u_0 )$.  
Therefore, there exist $\underline{b} \in \mathcal{A}^2$, $\underline{b}' \in \mathcal{A}'^2$ such that, writing 
\begin{equation*}
T_{ \underline{b} }^{\underline{\omega}_0} T'_{ \underline{b}' } (u_0) = \hat{u}_0, 
\end{equation*}
we have $\hat{u}_0 \in \mathcal{L}^0$. Let 
\begin{equation*}
T_{ \underline{b} }^{\underline{\omega}_0} T'_{ \underline{b}' } (u) = \hat{u}.
\end{equation*}
Write $\hat{u}_0 = (\hat{r}_0, \hat{t}_0)$ and 
$\hat{u} = (\hat{r}, \hat{t})$. 
It is easy to see that $| \hat{r} - \hat{r}_0 |$ and 
$| \hat{t} - \hat{t}_0 |$ 
are both of order $\rho^{3/2}$. Therefore, we obtain $\hat{u} \in \mathcal{L}$. 
\end{proof}

\section{Construction of the set $E$ and $L(r)$}\label{construct}

\subsection{Construction of $E$}
Let $\mu$ (resp. $\mu'$) be the uniform probability self-similar measures of $K$ (resp. $K'$). 
Denote the push-forward of $\mu'$ under the map $x \mapsto rx$ by $\mu'_{r}$. 
Take $M > 0$. We assume that $M$ is sufficiently large. 
Kaufman's proof of Marstrand's theorem tells us that 
the measure $\Pi ( \mu \times \mu'_r )$ is 
absolutely continuous with respect to Lebesgue measure 
for a.e. $r \in (M^{-1}, M)$, with $L^2$-density $\chi_{ r }$ 
satisfying 
\begin{equation*}
\int_{ ( M^{-1}, M ) } \| \chi_{ r } \|_{L^2}^2 \, d r < c_{4}. 
\end{equation*}
See, for example, section 4 in \cite{PalisTakens}. 
Define 
\begin{equation*}
E = \left\{ r \in (M^{-1}, M) : \| \chi_{r} \|_{L^2}^2 < c_5 \right\}, 
\end{equation*}
where $c_5 > 0$ is a sufficiently large constant so that 
$| (M^{-1}, M) \setminus E | < \epsilon / 2$.

\subsection{Construction of $L(r)$}
In this section we construct the set $L(r)$. 
Let $a_1 \in \mathcal{A}_1, a_2 \in \mathcal{A}_2$, $a'_1, a'_2 \in \mathcal{A}'$ 
and $\omega = ( \gamma, \delta )$.  
We denote the interval $f_{a_1}^{\omega}(I)$ by $I^{\gamma, \delta}(a_1)$ and the interval  
$f_{a_1}^{\omega} \circ f_{a_2}(I)$ by $I^{\gamma, \delta}(a_1 a_2)$. 
Also, we denote the interval $r f'_{a'_1}(I)$ by $I_r'(a'_1)$ and the interval 
$r f'_{a'_1} \circ f'_{a'_2}(I)$ by $I_r'(a'_1 a'_2)$. 
Let $u = ( r, t ) \in Q$. 
For 
\begin{equation*}
( \hat{r}, \hat{t} ) = T_{a_1}^{\omega} T'_{a'_1}(u), 
\end{equation*}
we denote  
$\hat{t}$ by $\pos_{t} \left( I^{ \gamma, \delta } (a_1 ) \times I_r' (a'_1) \right)$. 
Notice that 
\begin{equation*}
\hat{r} = \frac{ | I_r' ( a'_1 ) | }{ | I^{\gamma, 0}(a_1) | }. 
\end{equation*}
Define $\pos_{t} \left( I^{ \gamma, \delta } (a_1 a_2) \times I_r'(a'_1 a'_2) \right)$ analogously. 
With $c_{6} > 0$ conveniently small, to be chosen later, let 
\begin{equation*}
N = c^{2}_{6} \rho^{-\frac{1}{2} ( d + d' - 1 ) }. 
\end{equation*}
For $r \in E$, 
we define $L(r)$ to be the set of points 
$t \in (-1, M)$ such that the following holds 
($c_7 > 0$ is a sufficiently small constant to be chosen later) : 
there exist mutually distinct words 
$a_1^i \in \mathcal{A}_1 \ (i = 1, 2, \cdots, N)$, words $a'^i_1 \in \mathcal{A}' \ (i = 1, 2, \cdots, N)$ 
and the sets 
$\Phi_{i} \subset (-\epsilon, \epsilon)$ with $| \Phi_{i} | > c_{7}$ 
such that for all $a^i_1$, $a'^i_1$ and $\gamma \in \Phi_{i}$, there exist 
$a^{i}_2 \in \mathcal{A}_2$ and $a'^i_2 \in \mathcal{A}'$ 
such that 
\begin{equation*}
\frac{ | I_r'(a'^i_1 a'^i_2) | }{ | I^{\gamma, 0}( a^i_1 a^i_2 ) | }
 \in E \text{ \ and \ } 
| \pos_{t} \left( I^{\gamma, 0}( a_1^i a_2^i ) \times I_r'( a'^i_1 a'^i_2 ) \right) | \leq M. 
\end{equation*}
In the next section, we will prove the following estimate:
\begin{prop}\label{nice_estimate}
If $c_{6} > 0$ is sufficiently small, there exists $c_{8} > 0$ such that 
$| L(r) | > c_{8}$ for all $r \in E$.  
\end{prop}

\subsection{Projections of the rectangles $I(a) \times I_r'(a')$}
Let $r \in E$ and let $\mathcal{B} \subset \mathcal{A} \times \mathcal{A}'$ be such that 
$| \mathcal{B}  | > | \mathcal{A} \times \mathcal{A}' | / 24$. 
For $a \in \mathcal{A}$ and $a' \in \mathcal{A}'$, we have  
\begin{equation}\label{measure}
c_{9}^{-1} \rho^{\frac{1}{2}( d + d' )} < \mu \times \mu_r' ( I(a) \times I'_r(a') ) < 
c_{9} \rho^{\frac{1}{2} (d + d')}. 
\end{equation}
Write $\mathcal{J}( a, a' ) := \Pi ( I( a ) \times I_r'(a') )$.
Then 
\begin{equation*}
c_{9}^{-1} \rho^{1/2} < | \mathcal{J}(a, a') | < c_{9} \rho^{1/2}. 
\end{equation*}
We call $(a, a') \in \mathcal{B}$ \emph{$(\mathcal{B}, r)$-good} if there are no more than 
$c_{6}^{-1} \rho^{-\frac{1}{2}( d + d' - 1 )}$ intervals 
$\mathcal{J}( \tilde{a}, \tilde{a}' )$ $(  ( \tilde{a}, \tilde{a}' ) \in \mathcal{B} )$ 
whose centers are distant from the center of $\mathcal{J}(a, a')$ by less than $c_9^{-1} \rho^{ 1/2 }$. 
Call $(a, a') \in \mathcal{B}$ \emph{$( \mathcal{B}, r )$-bad} if it is not 
$( \mathcal{B}, r)$-good. 
Recall that, since $r \in E$, the measure $\Pi( \mu \times \mu_r' )$ has $L^2$-density $\chi_{r}$ which satisfies  
$\left\| \chi_{r} \right\|^2_{L^2} < c_5$.

\begin{lem}\label{bad_good}
The number of $( \mathcal{B}, r)$-bad pairs $(a, a')$ is less than 
\begin{equation*}
6 c_5 c_{6}  c^3_{9} \rho^{-\frac{1}{2}(d + d')}. 
\end{equation*}
In particular, if $c_6 > 0$ is sufficiently small, the number of 
$( \mathcal{B}, r)$-good pairs is at least 
$| \mathcal{B} | / 2$. 
\end{lem}

\begin{proof} 
Let $(a, a') \in \mathcal{B}$ be $(\mathcal{B}, r)$-bad. 
Then we have 
\begin{equation*}
\begin{aligned}
\int_{3 \mathcal{J}( a, a')} \chi_{r} 
&\geq c^{-1}_{9} \rho^{ \frac{1}{2}(d + d')} \cdot c_{6}^{-1} \rho^{ -\frac{1}{2} (d + d' -1) } \\
&= c^{-1}_{6} c^{-1}_{9} \rho^{1/2} > \frac{1}{3} c_{6}^{-1} c_{9}^{-2} | 3 \mathcal{J}( a, a' ) |, 
\end{aligned}
\end{equation*}
where $3 \mathcal{J}(a, a')$ is the interval of the same center as $\mathcal{J}(a, a')$ and length $3| \mathcal{J}(a, a') |$. 
By the Cauchy-Schwarz inequality, 
\begin{equation*}
\begin{aligned}
\frac{1}{3} c_{6}^{-1} c_{9}^{-2} | 3 \mathcal{J}( a, a' ) | \int_{3 \mathcal{J}( a, a' )} \chi_{r}  &\leq 
\left( \int_{3 \mathcal{J}( a, a' )} \chi_{r}  \right)^{2} \\
&\leq | 3 \mathcal{J} ( a, a' ) | \int_{3 \mathcal{J}( a, a' )} \chi_{r}^{2}, 
\end{aligned}
\end{equation*}
and thus 
\begin{equation*}
\int_{3 \mathcal{J}( a, a' )} \chi_{r}^{2} \geq 
\frac{1}{3} c_{6}^{-1} c_{9}^{-2}  \int_{3 \mathcal{J}( a, a' )} \chi_{r}. 
\end{equation*}
Let $\mathcal{J}^{*}$ be the union over all $( \mathcal{B}, r )$-bad pairs  
$(a, a')$ of the intervals $3\mathcal{J}( a, a' )$. One can 
extract a subfamily of intervals whose union is $\mathcal{J}^{*}$ and does not cover any point more than twice. 
Then we obtain 
\begin{equation*}
\int_{\mathcal{J}^*} \chi_{r}^{2}  \geq 
\frac{1}{6} c_{6}^{-1} c_{9}^{-2}  \int_{ \mathcal{J}^{*} } \chi_{r}.
\end{equation*}
Therefore,   
\begin{equation*}
 \int_{ \mathcal{J}^{*}} \chi_{r} \leq 6 c_5 c_{6} c_{9}^2. 
\end{equation*}
As $\mathcal{J}^{*}$ contains $\mathcal{J}( a, a' )$ 
for all $( \mathcal{B}, r )$-bad pairs $(a, a')$, together with (\ref{measure}) 
the estimate of the lemma follows. 
\end{proof}

Lemma \ref{bad_good} implies the following: 
 
\begin{lem}\label{proj}
\begin{equation*}
\Big| \bigcup_{ (a, a') \in \mathcal{B} } \mathcal{J}(a, a') \Big| > c_6 c_{10}. 
\end{equation*}
\end{lem} 
 
\begin{proof}
We have 
\begin{equation*}
\begin{aligned}
\Big| \bigcup_{ (a, a') \in \mathcal{B} } \mathcal{J}(a, a') \Big| &> 
\frac{1}{2} | \mathcal{B} | \cdot 
\left( c^{-1}_6 \rho^{-\frac{1}{2}(d + d' - 1)} \right)^{-1} \cdot c_9^{-1} \rho^{1/2} \\  
&> \frac{1}{2} \cdot \frac{1}{24} \cdot c_9^{-1} \rho^{-\frac{1}{2} d} \cdot c_9^{-1} \rho^{-\frac{1}{2} d'}
\cdot c_6 \rho^{\frac{1}{2}(d + d' - 1)} \cdot c_9^{-1} \rho^{1/2} 
= \frac{1}{48} c_6 c^{-3}_{9}. 
\end{aligned}
\end{equation*}
\end{proof}

\subsection{Proof of Proposition \ref{nice_estimate}} 
In this section we prove Proposition \ref{nice_estimate}. 
We fix $r \in E$ for the rest of the section. We assume that $r \geq 1$. 
The case of $r < 1$ is completely analogous. 
Let $a_1 \in \mathcal{A}^l_1$, $a'_1 \in \mathcal{A}'$ and $\gamma \in (-\epsilon, \epsilon)$.  
Note that 
\begin{equation*}
M^{-1} < \frac{ | I_r'(a'_1) | }{ | I^{\gamma, 0}( a_1 ) | } < M. 
\end{equation*}
Define 
\begin{equation*}
\Lambda_{ a_1, a'_1, \gamma } = 
\Big\{ ( a_2, a'_2 ) \in \mathcal{A}^l_2 \times \mathcal{A}' : 
\frac{ | I'_r(a'_1 a'_2) | }{ | I^{\gamma, 0}( a_1 a_2 ) | } \in E \Big\} 
\end{equation*}
and 
\begin{equation*}
\Phi^{*}_{a_1, a'_1} = \Big\{ \gamma \in (-\epsilon, \epsilon) : 
| \Lambda_{a_1, a'_1, \gamma} | > \frac{1}{4} | \mathcal{A}^l_2 \times \mathcal{A}' |  \Big\}. 
\end{equation*}
The following lemma is immediate. 
\begin{lem}
For any $a_1 \in \mathcal{A}^l_1$ and $a'_1 \in \mathcal{A}'$, 
we have $| \Phi^{*}_{a_1, a'_1} | >  \epsilon$. 
\end{lem}

\begin{proof}
By the construction of $E$, we have 
\begin{equation*}
\Big| \Big\{ \gamma \in (-\epsilon, \epsilon) : \frac{ | I_r'(a'_1 a'_2) | }{ | I^{\gamma, 0}( a_1 a_2 ) | }
\in E \Big\} \Big| > \frac{3}{2} \epsilon 
\end{equation*}
for all $a_2 \in \mathcal{A}^l_2$ and $a'_2 \in \mathcal{A}'$. 
Let $\psi$ be the sum, over $(a_2, a'_2) \in \mathcal{A}^l_2 \times \mathcal{A}'$, 
of the characteristic functions of 
\begin{equation*}
\Big\{ \gamma \in (-\epsilon, \epsilon) : 
\frac{ | I_r'(a'_1 a'_2) | }{ | I^{\gamma, 0}( a_1 a_2 ) | }
\in E \Big\}.
\end{equation*} 
Note that $0 \leq \psi \leq | \mathcal{A}^l_2 \times \mathcal{A}' |$ and 
\begin{equation*}
\gamma \in \Phi^{*}_{a_1, a'_1} \iff \psi( \gamma ) > \frac{1}{4} | \mathcal{A}^l_2 \times \mathcal{A}' |.
\end{equation*}   
Therefore, we have
\begin{equation*}
\begin{aligned}
\frac{3}{2} \epsilon \cdot | \mathcal{A}^l_2 \times \mathcal{A}' | < 
\int_{(-\epsilon, \epsilon)} \psi &= \int_{\Phi^{*}_{a_1, a'_1}} \psi + 
\int_{( -\epsilon, \epsilon ) \setminus \Phi^{*}_{a_1, a'_1}} \psi \\
&< | \mathcal{A}^l_2 \times \mathcal{A}' | \cdot | \Phi^{*}_{a_1, a'_1} | + 
\frac{1}{4} | \mathcal{A}^l_2 \times \mathcal{A}' |  \cdot 2 \epsilon. 
\end{aligned}
\end{equation*}
The claim follows from this. 
\end{proof}

For $a_1 \in \mathcal{A}^l_1$, $a'_1 \in \mathcal{A}'$, let  
\begin{equation*}
\Phi^{**}_{a_1, a'_1} = 
\Phi^{*}_{a_1, a'_1} \cap 
\Big\{ \gamma \in (-\epsilon, \epsilon) : \frac{ | I_r'( a'_1 ) | }{ | I^{\gamma, 0}( a_1 ) | } \in E \Big\}. 
\end{equation*}
By the above lemma, 
we have $| \Phi^{**}_{a_1, a'_1} | > \epsilon / 2$. 
For $a_1 \in \mathcal{A}^l_1$, 
$a_2 \in \mathcal{A}^l_2$, $a'_1, a'_2 \in \mathcal{A}'$ 
and $\gamma \in \Phi^{**}_{a_1, a'_1}$, we denote  
\begin{equation*}
J^{\gamma}( a_1 a_2, a'_1 a'_2 ) = \Pi \left( I^{ \gamma, 0 }(a_1 a_2) \times I_r'(a'_1 a'_2) \right), 
\end{equation*}
and for $a_1 \in \mathcal{A}^l_1$, $a'_1 \in \mathcal{A}'$ and $\gamma \in \Phi^{**}_{a_1, a'_1}$, write 
\begin{equation*}
J^{\gamma}(a_1, a'_1) = 
\bigcup_{ (a_2, a'_2) \in \Lambda_{a_1, a'_1, \gamma} } J^{\gamma}( a_1 a_2, a'_1 a'_2 ). 
\end{equation*}
By Lemma \ref{proj}, we have 
\begin{equation*}
| J^{\gamma}(a_1, a'_1) | > c_6 c_{10} \rho^{1/2}. 
\end{equation*} 
For $t \in \mathbb{R}$, let 
\begin{equation*}
\Phi_{a_1, a'_1, t} = \left\{ \gamma \in \Phi_{a_1, a'_1}^{**} : t \in J^{\gamma}(a_1, a'_1) \right\}. 
\end{equation*}
Write 
\begin{equation*}
J_{a_1, a'_1} = \left\{ t \in \mathbb{R} : \left| \Phi_{a_1, a'_1, t} \right| > 
\frac{1}{2} c_6 c_9^{-1} c_{10} | \Phi^{**}_{a_1, a'_1} | \right\}. 
\end{equation*}
Note that we have 
\begin{equation*}
\begin{aligned}
\frac{1}{2} c_{6} c_9^{-1} c_{10} | \Phi_{a_1, a'_1}^{**} | 
&> \frac{1}{2} c_{6} c_9^{-1} c_{10} \cdot \frac{1}{2} \epsilon \\
&=: c_7. 
\end{aligned}
\end{equation*}

\begin{lem}
We have 
\begin{equation*}
| J_{a_1, a'_1} | > \frac{1}{2} c_6 c_{10} \rho^{1/2}.
\end{equation*} 
\end{lem}

\begin{proof}
Let us integrate the characteristic function of 
\begin{equation*}
\big\{ (t, \gamma) : t \in J^{\gamma}(a_1, a'_1) \text{ for some } 
\gamma \in \Phi^{**}_{a_1, a'_1} \big\}
\end{equation*}
over $\{ (t, \gamma) : t \in \mathbb{R}, \gamma \in ( -\epsilon, \epsilon ) \}$. By Fubini's theorem, we have
\begin{equation*}
\begin{aligned}
| \Phi^{**}_{a_1, a'_1} | \cdot c_6 c_{10} \rho^{1/2} 
&< \int_{\gamma \in (-\epsilon, \epsilon)} \int_{t \in \mathbb{R}} \\ 
&= \int_{t \in \mathbb{R}} \int_{\gamma \in (-\epsilon, \epsilon)} \\ 
&= \int_{t \in J_{a_1, a'_1}} \int_{\gamma \in (-\epsilon, \epsilon)}  + 
\int_{t \in \mathbb{R} \setminus J_{a_1, a'_1} }  \int_{\gamma \in (-\epsilon, \epsilon) } \\
&< | J_{a_1, a'_1} | | \Phi^{**}_{a_1, a'_1} | + c_9 \rho^{1/2} \cdot 
\frac{1}{2} c_{6} c_{9}^{-1} c_{10} | \Phi^{**}_{a_1, a'_1} |. 
\end{aligned}
\end{equation*}
The claim follows from this. 
\end{proof}

Let $\psi$ be the sum, over $( \mathcal{A}^l_1 \times \mathcal{A}', r )$-good pairs 
$(a_1, a'_1)$, of the characteristic functions of 
$J_{a_1, a'_1}$. Note that $\supp \psi \subset ( -1, M )$ and  
\begin{equation*}
0 \leq \psi \leq c_{6}^{-1} \rho^{ -\frac{1}{2} (d + d' - 1)}. 
\end{equation*}
Let $\mathcal{D} = \{  \psi \geq  c^2_{6} \rho^{ -\frac{1}{2} (d + d' - 1)}  \}$.
Then we have 
\begin{equation*}
\begin{aligned}
| \mathcal{D} | \cdot c_{6}^{-1} \rho^{ -\frac{1}{2} (d + d' - 1)} + 
(1 + M) \cdot  c^2_{6} \rho^{ -\frac{1}{2} (d + d' - 1)} 
&\geq \int_{ \mathcal{D} } \psi + \int_{ (-1, M) \setminus \mathcal{D} } \psi  \\
&= \int \psi \\
&\geq \frac{1}{2} c_6 c_{10} \rho^{1/2} \cdot \frac{1}{3} \cdot \frac{1}{2} \cdot c_{9}^{-2} \rho^{-\frac{1}{2} (d + d')}. 
\end{aligned}
\end{equation*}
Take $c_6 > 0$ small enough so that 
\begin{equation*}
(1 + M) c_6^2 < \frac{1}{24} c_{6} c_{9}^{-2} c_{10}. 
\end{equation*}
holds. 
Then we obtain 
\begin{equation*}
| \mathcal{D} | \geq \frac{1}{24} c_{6} c_{9}^{-2} c_{10} \cdot c_{6} =: c_{8}. 
\end{equation*}
Since $\mathcal{D} \subset L(r)$, we have proved that 
\begin{equation*}
| L(r) | \geq c_{8}. 
\end{equation*}  
This concludes the proof of Proposition \ref{nice_estimate}.

\section{Proof of the key Proposition}\label{key_prop}

\subsection{Proof of Proposition \ref{key_lem}} 
In this section we prove Proposition \ref{key_lem}. 
Fix $(r, t) \in \mathcal{L}^1$. 
Let $( \tilde{r}, \tilde{t} ) \in \mathcal{L}^0$ be such that $| r - \tilde{r} | < \rho$ and 
$| t - \tilde{t} | < \rho$. 

Then, there exist mutually distinct words 
$a_1^i \in \mathcal{A}_1 \ (i = 1, 2, \cdots, N)$, words $a'^i_1 \in \mathcal{A}' \ (i = 1, 2, \cdots, N)$ 
and the sets 
$\Phi_{i} \subset (-\epsilon, \epsilon)$ with $| \Phi_{i} | > c_{7}$ 
such that for all $a^i_1$, $a'^i_1$ and $\tilde{ \gamma } \in \Phi_{i}$, there exist 
$a^{i}_2 \in \mathcal{A}_2$ and $a'^i_2 \in \mathcal{A}'$ 
such that 
\begin{equation*}
\frac{ | I_{ \tilde{r} }'(a'^i_1 a'^i_2) | }{ | I^{ \tilde{\gamma}, 0}( a^i_1 a^i_2 ) | }
 \in E \text{ \ and \ } 
\big| \pos_{ \tilde{t} } \big( I^{ \tilde{\gamma}, 0}( a_1^i a_2^i ) \times I_{ \tilde{r} }'( a'^i_1 a'^i_2 ) \big) \big| \leq M. 
\end{equation*}

Let 
\begin{equation*}
\mathcal{A}_3 = \left\{ a^1_1, a^2_1, \cdots, a^N_1 \right\}.
\end{equation*}
Write 
\begin{equation*}
\begin{aligned}
\Omega &= \left( ( -\epsilon, \epsilon ) \times (-1, 1) \right)^{ \mathcal{A}_3 } 
\times \left( ( -\epsilon, \epsilon ) \times (-1, 1) \right)^{ \mathcal{A}_1 \setminus \mathcal{A}_3 } ,  \\
\underline{\omega} &= ( \underline{\omega}', \underline{\omega}'' ), \text{ and \ } 
 \underline{\omega}' = \left( \omega_1, \omega_2, \cdots, \omega_N \right). 
\end{aligned}
\end{equation*}
Denote $\omega_i = ( \gamma_i, \delta_i )$. 
By Fubini's Theorem, Proposition \ref{key_lem} follows from the following claim:

\begin{claim}
There exists $c'_2 > 0$ such that for any $a_1^i \in \mathcal{A}_3$ and $a'^{i}_1$, 
\begin{equation*}
\left|  \left\{  \omega_i : \exists a_2^i \in \mathcal{A}_2, a_2'^{i} \in \mathcal{A}' \text{ s.t. } 
T_{a_1^i a_2^i}^{\omega_i} T'_{ a'^i_1 a'^i_2 } (r, t) \in \mathcal{L}^0 
 \right\}  \right| > c'_2. 
\end{equation*}
\end{claim}

\begin{proof}[Proof of the claim]
Take $a_1^i \in \mathcal{A}_3$, $a'^i_1$ and $\tilde{\gamma} \in \Phi_i$. 
Let 
$\displaystyle{ \gamma = ( 1 + \tilde{\gamma} ) \frac{ r }{ \tilde{r} } - 1 }$.  
Note that 
\begin{equation}\label{maane}
\frac{1 + \tilde{\gamma}}{ 1 + \gamma } = \frac{ \tilde{r} }{r}, 
\end{equation}
and $| \gamma - \tilde{\gamma} |$ is of order $\rho$. 
Let $a_2^i \in \mathcal{A}_2$, $a'^i_2 \in \mathcal{A}'$ be such that 
\begin{equation*}
\frac{ | I_{ \tilde{r} }' (a'^i_1 a'^i_2) | }{ | I^{ \tilde{\gamma}, 0}( a^i_1 a^i_2 ) | }
 \in E \text{ \ and \ } 
\big| \pos_{ \tilde{t} } \big( I^{ \tilde{\gamma}, 0}( a_1^i a_2^i ) \times I_{ \tilde{r} }' ( a'^i_1 a'^i_2 ) \big) \big| \leq M. 
\end{equation*}
By (\ref{maane}), we have 
\begin{equation*}
\frac{ | I_{ \tilde{r} }' (a'^i_1 a'^i_2) | }{ | I^{ \tilde{\gamma}, 0}( a^i_1 a^i_2 ) | } = 
\frac{ | I_{r}' (a'^i_1 a'^i_2) | }{ | I^{ \gamma, 0}( a^i_1 a^i_2 ) | }. 
\end{equation*}
Therefore, we have 
\begin{equation*}
\frac{ | I_r'(a'^i_1 a'^i_2) | }{ | I^{ \gamma, \delta}( a^i_1 a^i_2 ) | } \in E
\end{equation*}
for all $\delta \in (-1, 1)$. It is easy to see that 
\begin{equation*}
\big| \pos_{ t } \big( I^{\gamma, 0} (a_1^{i} a_2^{i} ) \times I_r'( a'^i_1 a'^i_2 ) \big) - 
\pos_{ \tilde{t} } \big( I^{ \tilde{\gamma} , 0} (a_1^{i} a_2^{i} ) \times I_{ \tilde{r} }'( a'^i_1 a'^i_2 ) \big)  \big| 
\end{equation*}
is of order $1$. 
Therefore, 
\begin{equation*}
\big| \pos_{ t } \big( I^{\gamma, 0} (a_1^{i} a_2^{i} ) \times I_r'( a'^i_1 a'^i_2 ) \big) \big| 
\end{equation*}
 is also of order $1$. 
It follows that, by taking $c_1 > 0$ large enough, we obtain 
\begin{equation*}
\big| \big\{ \delta \in (-1, 1) : \pos_{ t } \big( I^{\gamma, \delta} (a_1^{i} a_2^{i} ) \times I_r'( a'^i_1 a'^i_2 ) \big)  
\in L( r )  \big\} \big| > c_2''. 
\end{equation*}
Therefore, 
\begin{equation*}
\begin{aligned}
\left|  \left\{  \omega_i : \exists a_2^i \in \mathcal{A}_2, a_2'^{i} \in \mathcal{A}' \text{ s.t. } 
T_{a_1^i a_2^i}^{\omega_i} T'_{ a'^i_1 a'^i_2 } (r, t) \in \mathcal{L}^0 
 \right\} \right| 
&> |  \Phi_i  | / 2 \cdot c_2'' / 2 \\
&=: c_2'. 
\end{aligned}
\end{equation*}
\end{proof}

\section{The case of $K = K'$}\label{proof2}
\subsection{Proof of Theorem \ref{thm2}}
In this section we explain how to modify the proof of Theorem \ref{thm1} to 
prove Theorem \ref{thm2}. 
We assume $\mathcal{F} = \{ f_a \}_{a \in \mathcal{A}}$ to satisfy the following:  
\begin{itemize}
\item[(i)] $c_0^{-1} \rho^{1/2} < | I(a) | < c_0 \rho^{1/2}$ for all $a \in \mathcal{A}$; 
\item[(ii)] there exist $\mathcal{A}^{l}$, $\mathcal{A}^{s} \subset \mathcal{A}$ that satisfy   
$\mathcal{A}^{l} \sqcup \mathcal{A}^{s} = \mathcal{A}$, $| \mathcal{A}^l | = | \mathcal{A}^s | = | \mathcal{A} | / 2$ and 
$\min_{a \in \mathcal{A}^l } | I(a) | > \max_{a \in \mathcal{A}^s } | I( a ) |$. 
\end{itemize} 
We take $\mathcal{A}^l_1$, $\mathcal{A}^l_2 \subset \mathcal{A}^l$ and 
$\mathcal{A}^s_1$, $\mathcal{A}^s_2 \subset \mathcal{A}^s$ as in the proof of Theorem \ref{thm1}.  

Next we modify the definition of $L(r)$. Assume that $r \geq 1$. The case of $r < 1$ is analogous. 
We define $L(r)$ to be the set of points 
$t \in (-1, M)$ such that the following holds: 
there exist mutually distinct words 
$a_1^i \in \mathcal{A}_1 \ (i = 1, 2, \cdots, N)$, words $a'^i_1 \in \mathcal{A}_2 \ (i = 1, 2, \cdots, N)$ 
and the sets 
$\Phi_{i} \subset (-\epsilon, \epsilon)$ with $| \Phi_{i} | > c_{7}$ 
such that for all $a^i_1$, $a'^i_1$ and $\gamma \in \Phi_{i}$, there exist 
$a^{i}_2 \in \mathcal{A}_2$ and $a'^i_2 \in \mathcal{A}_2$ 
such that 
\begin{equation*}
\frac{ | I_r(a'^i_1 a'^i_2) | }{ | I^{\gamma, 0}( a^i_1 a^i_2 ) | }
 \in E \text{ \ and \ } 
| \pos_{t} \left( I^{\gamma, 0}( a_1^i a_2^i ) \times I_r( a'^i_1 a'^i_2 ) \right) | \leq M. 
\end{equation*}
The rest is analogous to the proof of Theorem \ref{thm1}.

\section*{Acknowledgements}
The author is grateful to Boris Solomyak for many helpful discussions and comments.

\end{document}